\documentclass[twoside,11pt]{article}
\usepackage{graphicx,amsmath,latexsym,amssymb,amsthm}
 \usepackage{epstopdf}
\usepackage[colorlinks=black,urlcolor=black]{hyperref}
\font\chuto=cmbx10 at 16pt \font\kamy=lcmssb8
at 9pt

\newtheorem{theorem}{Theorem}[section]

\newtheorem{definition}[theorem]{Definition}
\newtheorem{corollary}[theorem]{Corollary}
\newtheorem{lemma}[theorem]{Lemma}
\newtheorem{example}[theorem]{Example}
\newtheorem{remark}[theorem]{Remark}

\numberwithin{equation}{section}

\begin{document}
\vskip1.5cm

\centerline {\bf \chuto On Properties of Geodesic Semilocal   E-Preinvex Functions
}

\vskip.2cm

\centerline {\bf \chuto }


\vskip.8cm \centerline {\kamy Adem Kili\c{c}man  $ ^{a} $ and Wedad Saleh  $^b,$\footnote{{\tt Corresponding author email: wed\_10\_777@hotmail.com ( Wedad Saleh)}}}

\vskip.5cm

\centerline {$^a$ Department of Mathematics, Putra University of Malaysia (UPM),Serdang, Malaysia   }
 \centerline {$^{b}$ Department of Mathematics, Taibah University, Al- Medina, Saudi Arabia } 


\vskip.5cm \hskip-.5cm{\small{\bf Abstract :} The authors define 
	 a class of functions on Riemannian manifolds, which is called geodesic semilocal  E-preinvex functions, as a generalization of geodesic  semilocal  E-convex and geodesic semi  E-preinvex functions and some of its properties are established. Furthermore, a nonlinear fractional multiobjective programming is considered, where the functions involved are geodesic E-$ \eta $-semidifferentiability, sufficient optimality conditions are obtained. A dual is formulated and duality results are proved using concepts of geodesic semilocal  E-preinvex functions, geodesic  pseudo-semilocal   E-prinvex functions and geodesic quasi-semilocal  E-preinvex functions.

\vskip0.3cm\noindent {\bf Keywords :}
Generalized convexity, Riemannian geometry, Duality

\noindent{\bf 2000 Mathematics Subject Classification : 52A41; 53B20;   90C46   } 

\hrulefill


\section{Introduction}
\hskip0.6cm
  Convexity and generalized covexity play a significant role in many fields, for example, in biological system, economy, optimization, and so on \cite{kS,GrinalattLinnainmaa2011,RuelAyres1999}. 
 
  \vspace{1.0mm}
   Generalized convex functions, labelled as semilocal convex functions were introduced by Ewing  \cite{Ewing} by using more general semilocal perinvexity and $ \eta- $ semidifferentiability. After that optimality conditions for weak vector minima was given  \cite{Preda1997}. Also, optimality conditions and duality results for a nonlinear fractioal involving $ \eta- $ semidifferentiability were established  \cite{Preda2003}.
   
 Furthermore,some optimality conditions and duality results for semilocal E-convex programming were established \cite{Hu2007}. E-convexity was extedned to E-preinvexity \cite{FulgaPreda}. Recently, semilocal E-prenivexity (SLEP) and some of its applications were introdued \cite{Jiao2011,Jiao2012,Jiao2013}.

 \vspace{1.0mm}
  Generalized convex functions in manifolds such as Riemannian manifolds were studied by many authors; see \cite{Agarwal,BG,Ferrara,Mordukhovch2011}. Udrist \cite{Udriste1994} and Rapcsak  \cite{Rapcsak1997} considered a generalization of convexity called geodesic convexity.  In 2012, geodesic E-convex (GEC) sets and geodesic E-convex (GEC)functions on Riemannian manifolds were studied \cite{IAA2012}. Moreover, geodesic semi E-convex (GsEC) functions were introduced \cite{Iqbal}. Recently, geodesic strongly E-convex (GSEC) functions were introduced and discussed some of their properties \cite{AW}.
 \section{Geodesic Semilocal  E-Preinvexity}
  \hskip0.6cm
 \begin{definition}\label{df}
  A nonempty set $ B \subset \aleph $ is said to be
 \begin{enumerate}
 	\item geodesic E-invex (GEI) with respect to $ \eta $ if there is exactly one geodesic $ \gamma_{E(\kappa_{1}), E(\kappa_{2})}:\left[0,1 \right]\longrightarrow \aleph  $ such that 
 	\begin{eqnarray*}
 	\gamma_{E(\kappa_{1}), E(\kappa_{2})}(0)=E(\kappa_{2}),\acute{\gamma}_{E(\kappa_{1}), E(\kappa_{2})}=\eta(E(\kappa_{1}),E(\kappa_{2}), \gamma_{E(\kappa_{1}), E(\kappa_{2})}(t)\in B,
 	\end{eqnarray*}
 $  \forall \kappa_{1},\kappa_{2}\in B$ and $t\in[0,1].  $
 	\item a geodesic local E-invex (GLEI) respect to $ \eta $, if  there is $ u(\kappa_{1},\kappa_{2})\in\left(\left.0,1 \right]  \right.  $ such that $ \forall t\in [0,u (\kappa_{1},\kappa_{2})] $,
 	\begin{equation}\label{eq2}
 	\gamma_{E(\kappa_{1}),E(\kappa_{2})}(t)\in B \ \ \forall\kappa_{1},\kappa_{2}\in B. 
 	\end{equation} 
 	\item
 	 a geodesic local starshaped E-convex, if there is a map $ E $ such that corresponding to each pair of points $ \kappa_{1},\kappa_{2}\in A $, there is a maximal positive number $ u(\kappa_{1},\kappa_{2})\leq 1 $ such as  
 	\begin{equation}\label{eq1}
 	\gamma_{E(\kappa_{1}),E(\kappa_{2})}\in A, \ \ \forall t\in [0, u(\kappa_{1},\kappa_{2})] 
 	\end{equation}
 \end{enumerate}    
 \end{definition}

\begin{definition}
	  A function $ f: A\subset \aleph\longrightarrow \mathbb{R} $ is said to be 
	 \begin{enumerate}
	 	\item Geodesic E-preinvex (GEP) on $ A\subset \aleph$ with repect to $ \eta $ if $A$ is a GEI set and $$f\left(\gamma_{E(\kappa_{1}),E(\kappa_{2})}(t)\right)\leq t f(E(\kappa_{1}))+(1-t)f(E(\kappa_{2})) , \ \ \forall \kappa_{1},\kappa_{2}\in A, t\in[0,1]; $$
	 	\item  Geodesic semi E-preinvex (GSEP)  on $ A $ with respect to $ \eta $ if if $A$ is a GEI set and
	 	$$f\left(\gamma_{E(\kappa_{1}),E(\kappa_{2})}(t)\right)\leq t f(\kappa_{1})+(1-t)f(\kappa_{2}) , \ \ \forall \kappa_{1},\kappa_{2}\in A, t\in[0,1]. $$
	 	\item Geodesic Local E-preinvex (GLEP) on $ A\subset \aleph $ with respect to $ \eta $, if for any $ \kappa_{1},\kappa_{2}\in A $ there exists $ 0<v(\kappa_{1},\kappa_{2})\leq u(\kappa_{1},\kappa_{2}) $ such that $ A $ is a  GLEI  set and 
	 	$$f\left(\gamma_{E(\kappa_{1}),E(\kappa_{2})}(t)\right)\leq t f(E(\kappa_{1}))+(1-t)f(E(\kappa_{2})) , \ \ \forall t\in[0,v(\kappa_{1},\kappa_{2})].$$
	 \end{enumerate}
\end{definition}
\begin{definition}
	A function $ f:\aleph\longrightarrow \mathbb{R} $ is a geodesic semilocal E-convex ( GSLEC) on a geodesic local starshaped  E-convex set $ B\subset \aleph $ if for each pair of $ \kappa_{1},\kappa_{2}\in B $ ( with a maximal positive number $ u(\kappa_{1},\kappa_{2})\leq1 $ satisfying \ref{eq1}), there exists a positive number $ v(\kappa_{1},\kappa_{2})\leq u(\kappa_{1},\kappa_{2}) $ satisfying $$f\left(\gamma_{E(\kappa_{1}), E(\kappa_{2})}(t)\right)\leq t f(\kappa_{1})+(1-t)f(\kappa_{2}) , \ \ \forall t\in[0,v(\kappa_{1},\kappa_{2})].$$
\end{definition}
\begin{remark}
	Every GEI set with respect to $ \eta $ is a  GLEI set with respect to $ \eta $, where $ u(\kappa_{1},\kappa_{2})=1, \forall \kappa_{1},\kappa_{2}\in \aleph $. On the other hand, their coverses are not recessarily true and we can see that in the next example.
\end{remark}
\begin{example}
	Put $ A=\left[ \left. -4,-1\right)  \right. \cup[1,4] $, \begin{eqnarray*}
		E(\kappa) &=& \begin{cases}
			\kappa^{2} \ \ if\ \ \left|\kappa \right| \leq 2,\\   -1\ \ if \ \ \left|\kappa \right| > 2;
		\end{cases}
	\end{eqnarray*}
	\begin{eqnarray*}
		\eta (\kappa,\iota) &=& \begin{cases}
			\kappa-\iota \ \ if\ \ \kappa\geqslant 0, \iota\geqslant 0\ \ or\ \ \kappa\leq 0, \iota\leq0 ,\\   -1-\iota \ \ if \ \ \kappa>0, \iota\leq 0 \ \ or \ \ \kappa\geqslant0 ,\iota< 0, \\ 1-\iota \ \ if \ \ \kappa<0, \iota\geqslant 0 \ \ or\ \ \kappa\leq 0, \iota>0;
		\end{cases}
	\end{eqnarray*}
	\begin{eqnarray*}
		\gamma_{\kappa,\iota}(t) &=& \begin{cases}
			\iota+t(\kappa-l) \ \ if \ \ \kappa\geqslant 0, \iota\geqslant 0 \ \ or \ \ \kappa\leq 0, \iota\leq0 ,\\   \iota+t(-1-\iota) \ \ if \ \  \kappa>0, \iota\leq 0 \ \ or\ \ \kappa\geqslant0 ,\iota< 0, \\ \iota+t(1-\iota) \ \ if \ \ \kappa<0, \iota\geqslant 0 \ \ or \ \ \kappa\leq 0, \iota>0.
		\end{cases}
	\end{eqnarray*}
	Hence $ A $ is a  GLEI set with respect to $ \eta $.
	But, when $ \kappa=3, \iota=0 $, there is a $ t_{1}\in[0,1] $ such that $ \gamma_{E(\kappa),E(\iota)}(t_{1})=-t_{1} $, then if $ t_{1}=1 $, we obtain $\gamma_{E(\kappa),E(\iota)}(t_{1})\notin A  $.
\end{example}
\begin{definition}\label{de1}
	A function $ f: \aleph\longrightarrow \mathbb{R} $ is GSLEP  on $ B\subset \aleph $ with respect to $ \eta $ if for any $ \kappa_{1},\kappa_{2}\in B $, there is $ 0<v(\kappa_{1},\kappa_{2})\leq u(\kappa_{1},\kappa_{2})\leq1 $ such that $ B $ is a GLEI  set and
	\begin{equation}\label{eq3}
	f\left(\gamma_{E(\kappa_{1}),E(\kappa_{2})}(t)\right)\leq t f(\kappa_{1})+(1-t)f(\kappa_{2}) , \ \ \forall t\in[0,v(\kappa_{1},\kappa_{2})].
	\end{equation} 
	If $$f\left(\gamma_{E(\kappa_{1}),E(\kappa_{2})}(t)\right)\geqslant t f(\kappa_{1})+(1-t)f(\kappa_{2}) , \ \ \forall t\in[0,v(\kappa_{1},\kappa_{2})],$$ then $ f $ is GSLEP  on $ B $.
	
\end{definition}
\begin{remark}
	Any GSLEC  function is a GSLEP function. Also, any GSEP function with respect to $ \eta $ is a  GSLEP function. On the other hand, their converses are not necessarily true.
\end{remark}
The next example shows SLGEP, which is neither a GSLEC function nor a GSEP function.
\begin{example}
	Assume that $ E: \mathbb{R}\longrightarrow \mathbb{R} $ is given as
	 \begin{eqnarray*}
		E(m) &=& \begin{cases}
			0 \ \ if\ \ m<0,\\   1 \ \ if\ \  1<m\leq 2,\\ m \ \ if\ \ 0\leq m\leq1 \ \ or\ \ m>2
		\end{cases}
	\end{eqnarray*}
and the map $ \eta: \mathbb{R}\times \mathbb{R}\longrightarrow \mathbb{R} $ is defined as 
	\begin{eqnarray*}
		\eta (m,n) &=& \begin{cases}
			0 \ \ if\ \ m= n,\\  1-m \ \ if\ \ m\neq n ;
		\end{cases}
	\end{eqnarray*}
also,
	\begin{eqnarray*}
		\gamma_{m,n}(t) &=& \begin{cases}
			n \ \ if\ \ m= n,\\   n+t(1-m) \ \ if\ \  m\neq n. 
		\end{cases}
	\end{eqnarray*}
	Since $ \mathbb{R} $ is a geodesic local starshaped  E-convex set and a  geodesic local E-invex set with respect to $ \eta $. Assume that $ h: \mathbb{R}\longrightarrow \mathbb{R} $, where
	\begin{eqnarray*}
	h(m) &=& \begin{cases}
			0 \ \ if\ \ 1<m\leq 2,\\  1 \ \ if\ \  m>2, \\ -m+1 \ \ if\ \ 0\leq m\leq1,\\ -m+2 \ \ if\ \ m<0.
		\end{cases}
	\end{eqnarray*}
Then $ h $ is a GSLEP  on $ \mathbb{R} $ with respect to $ \eta $. However, when $ m_{0}=2, n_{0}=3 $ and for any $ v\in\left(\left.0,1 \right]  \right.  $, there is a sufficiently small $ t_{0}\in\left(\left.0,v \right]  \right.  $ such as
$$h\left(\gamma_{E(m_{0}),E(n_{0})}(t_{0}) \right)=1>(1-t_{0})=t_{0}h(m_{0})+(1-t_{0})h(n_{0}) .$$
Then $ h(m) $ is not a GSLEC function on $ \mathbb{R} $.

  \vspace{1.0mm}
Similarly, taking $ m_{1}=1, n_{1}=4 $, we have 
$$h\left(\gamma_{E(m_{1}),E(n_{1})}(t_{1}) \right)=1>(1-t_{1})=t_{1}h(m_{1})+(1-t_{1})h(n_{1}) .$$
for some $ t_{1}\in[0,1] $.\\

Hence $ h(m) $ is not a  GSEP function on $ \mathbb{R} $ with respect to $ \eta $.
\end{example}
\begin{definition}\label{de2}
A function $ h:S\subset \aleph\longrightarrow \mathbb{R} $ , where $ S $  a GLEI set, is said to be a geodesic quasi-semilocal E-preinvex (GqSLEP) (with respect to $ \eta $) if for all $ \kappa_{1},\kappa_{2}\in S $ satisfying $ h(\kappa_{1})\leq h(\kappa_{2}) $, there is a positive number $ v(\kappa_{1},\kappa_{2})\leq u(\kappa_{1},\kappa_{2}) $  such that 
$$h\left(\gamma_{E(\kappa_{1}),E(\kappa_{2})}(t)\right) \leq h(\kappa_{2}), \forall t\in[0,v(\kappa_{1},\kappa_{2})]. $$
\end{definition}
\begin{definition}\label{de3}
A function $ h:S\subset \aleph\longrightarrow \mathbb{R} $ , where $ S $  a GLEI set, is said to be a geodesic pseudo-semilocal E-preinvex ( GpSLEP) (with respect to $ \eta $) if for all $ \kappa_{1},\kappa_{2}\in S $ satisfying $ h(\kappa_{1})<h(\kappa_{2}) $, there are  positive numbers  $ v(\kappa_{1},\kappa_{2})\leq u(\kappa_{1},\kappa_{2}) $ and  $ w(\kappa_{1},\kappa_{2}) $ such that 
	$$h\left(\gamma_{E(\kappa_{1}),E(\kappa_{2})}(t)\right) \leq h(\kappa_{2})-t w(\kappa_{1},\kappa_{2}), \forall t\in[0,v(\kappa_{1},\kappa_{2})]. $$
\end{definition}
\begin{remark}
	Every  GSLEP on a GLEI set with respect to $ \eta $ is both a GqELEP  function and a  GpSLEP  function.
\end{remark}
\begin{definition}\label{de4}
A function $ h:S\longrightarrow \mathbb{R} $ 
is called  geodesic E-$ \eta $- semidifferentiable at $ \kappa^{*}\in S $ where $ S\subset \aleph $ is a GLEI set with respect to $ \eta $  if $ E(\kappa^{*})=\kappa^{*} $ and 
\begin{equation*}
h'_{+}\left(\gamma_{\kappa^{*},E(\kappa)}(t) \right)= \lim_{t\longrightarrow 0^{+}} \frac{1}{t}\left[h\left(\gamma_{\kappa^{*},E(\kappa)}(t)\right) -h(\kappa^{*}) \right],  
\end{equation*}
exists for every $ \kappa\in S. $.
\end{definition}
\begin{remark}
	\begin{enumerate}
		\item Let $ \aleph=\mathbb{R}^{n} $, then the geodesic E-$ \eta $- semidifferentiable is E-$ \eta $-semidifferentiable \cite{Jiao2011}.
		\item If $ \aleph=\mathbb{R}^{n} $ and $ E=I $, then the geodesic E-$ \eta $-semidifferentiable is the $ \eta $-semidifferentiablitiy \cite{Niculescu2007Optimality} .
		\item If $ \aleph=\mathbb{R}^{n} $ , $ E=I $  and $ \eta(\kappa,\kappa^{*})=\kappa-\kappa^{*} $, then geodesic E-$ \eta $-semidifferentiable is the semidifferentiability \cite{Jiao2011}.
	\end{enumerate}
\end{remark}
\begin{lemma}\label{lemma2}
	\begin{enumerate}
		\item Assume that $ h $ is a GSLEP (E-preconcave) and geodesic  E-$ \eta $-semidifferentiable at $ \kappa^{*}\in S\subset \aleph $, where $ S $ is a GLEI set with respect to $ \eta $. Then 
		$$h(\kappa)-h(\kappa^{*})\geqslant (\leq) h'_{+}(\gamma_{\kappa^{*},E(\kappa)}(t)), \forall \kappa\in S.$$
		\item Let $ h $ be  GqSLEP (GpSLEP) and geodesic   E-$ \eta $-semidifferentiable at $ \kappa^{*}\in S\subset \aleph $, where $ S $ is a LGEI set with respect to $ \eta $. Hence $$h(\kappa)\leq(<) h(\kappa^{*})\Rightarrow h'_{+}(\gamma_{\kappa^{*},E(\kappa)}(t))\leq (<)0, \forall \kappa\in S.$$
	\end{enumerate}
\end{lemma}
The above lemma is directly by using definitions (\ref{de1},\ref{de2},\ref{de3} and \ref{de4}).
\begin{theorem}
	Let $ f: S\subset \aleph\longrightarrow \mathbb{R} $ be a  GLEP function on a GLEI set $ S $ with  respect to $ \eta $, then $ f $ is a GSLEP  function iff $ f(E(\kappa))\leq f(\kappa),\forall \kappa\in S $.
\end{theorem} 
\begin{proof}
	Assume that $ f $ is a GSLEP function on set $ S $ with respect to $ \eta $, then $\forall \kappa_{1},\kappa_{2}\in S $, there is a positive number $ v(\kappa_{1},\kappa_{2})\leq u(\kappa_{1},\kappa_{2}) $ where
	$$f(\gamma_{E(\kappa_{1}),E(\kappa_{2})}(t))\leq tf(\kappa_{2})+(1-t)f(\kappa_{1}), t\in[0,v(\kappa_{1},\kappa_{2})].$$ 
	By letting $ t=0 $, then $ f(E(\kappa_{1}))\leq f(\kappa_{1}),\forall \kappa_{1}\in S $.\\
	
	Conversely, consider $ f $ is a GLEP function on a GLEI  set $ S $, then for any $ \kappa_{1},\kappa_{2}\in S $, there exist $ u(\kappa_{1},\kappa_{2}) \in \left(\left.0,1 \right]  \right. $ (\ref{eq2}) and $ v(\kappa_{1},\kappa_{2}) \in \left(\left.0,u(\kappa_{1},\kappa_{2}) \right]  \right. $
	such that 
	$$f(\gamma_{E(\kappa_{1}),E(\kappa_{2})}(t))\leq tf(E(\kappa_{1}))+(1-t)f(E(\kappa_{2})), t\in[0,v(\kappa_{1},\kappa_{2})].$$ 
	Since $ f(E(\kappa_{1})) \leq f(\kappa_{1}), \forall \kappa_{1}\in S$, then
	$$f(\gamma_{E(\kappa_{1}),E(\kappa_{2})}(t))\leq tf(\kappa_{1})+(1-t)f(\kappa_{2}), t\in[0,v(\kappa_{1},\kappa_{2})].$$ 
\end{proof}
\begin{definition}
	The set $ \omega=\left\lbrace(\kappa,\alpha):\kappa\in B\subset \aleph, \alpha\in \mathbb{R} \right\rbrace $ is said to be a GLEI set with respect to $ \eta $ corresponding to $ \aleph $ if there are two maps $ \eta, E $ and a maximal positive number $ u((\kappa_{1},\alpha_{1}), (\kappa_{2},\alpha_{2}))\leq 1 $, for each $ (\kappa_{1},\alpha_{1}), (\kappa_{2},\alpha_{2})\in \omega $ such that 
	$$ \left(\gamma_{E(\kappa_{1}),E(\kappa_{2})}(t),t\alpha_{1}+(1-t)\alpha_{2} \right)\in \omega, \forall t\in\left[0,u((\kappa_{1},\alpha_{1}), (\kappa_{2},\alpha_{2})) \right]. $$  
\end{definition}
\begin{theorem}\label{th1}
	Let $ B\subset  \aleph$ be a GLEI set with respect to $ \eta $. Then $ f $ is a GSLEP function on $ B $ with respect to $ \eta $ iff its epigraph $$ \omega_{f}=\left\lbrace (\kappa_{1},\alpha):\kappa_{1}\in B, f(\kappa_{1})\leq\alpha, \alpha\in \mathbb{R} \right\rbrace $$  is a GLEI set with respect to $ \eta $ corresponding to  $ \aleph $. 
\end{theorem}
\begin{proof}
	Suppose that $ f $ is a GSLEP on $ B $ with respect to $ \eta $ and $ (\kappa_{1},\alpha_{1}), (\kappa_{2},\alpha_{2})\in \omega_{f} $, then $ \kappa_{1},\kappa_{2}\in B, f(\kappa_{1})\leq \alpha_{1}, f(\kappa_{2})\leq \alpha_{2} $. By applying definition \ref{df}, we obtain 
	 $ \gamma_{E(\kappa_{1}),E(\kappa_{2})}(t)\in B, \forall t\in\left[0, u(\kappa_{1},\kappa_{2}) \right].  $\\
	 Moreover, 
	  there is a positive number $ v(\kappa_{1},\kappa_{2})\leq u(\kappa_{1},\kappa_{2}) $ such that 
	 $$f\left(\gamma_{E(\kappa_{1}),E(\kappa_{2})}(t), t\alpha_{1}+(1-t)\alpha_{2} \right)\in \omega_{f}, \forall t\in[0,v(\kappa_{1},\kappa_{2})]. $$
	 
	  \vspace{1.0mm} 
	 Conversely, if $ \omega_{f} $ is a GLEI set with respect to $ \eta $ corresponding to $ \aleph $ ,then for any points $ (\kappa_{1},f(\kappa_{1})) , (\kappa_{2},f(\kappa_{2}))\in \omega_{f}$, there is a maximal positive number $ u((\kappa_{1},f(\kappa_{1})), (\kappa_{2},f(\kappa_{2}))\leq 1 $ such that  $$\left( \gamma_{E(\kappa_{1}),E(\kappa_{2})}(t), tf(\kappa_{1}) +(1-t)f(\kappa_{2})\right) \in \omega_{f},\forall t\in\left[0, u((\kappa_{1},f(\kappa_{1})),(\kappa_{2},f(\kappa_{2}))) \right].$$ 
	 That is, $\gamma_{E(\kappa_{1}),E(\kappa_{2})}(t) \in B, $ $$f\left(\gamma_{E(\kappa_{1}),E(\kappa_{2})}(t) \right)\leq tf(\kappa_{1}) +(1-t)f(\kappa_{2}), \ \ t\in\left[0,u((\kappa_{1},f(\kappa_{1})),(\kappa_{2},f(\kappa_{2}))) \right]. $$ Thus, $ B $ is a GLEI set and $ f $ is a GSLEP function on $ B $.  
\end{proof}
\begin{theorem}
	If $ f $ is a GSLEP function on a  GLEI set $ B\subset \aleph $ with respect to $ \eta $ , then the level $  K_{\alpha}=\left\lbrace \kappa_{1}\in B: f(\kappa_{1})\leq \alpha \right\rbrace $ is a  GLEI set for any $ \alpha\in \mathbb{R} $.
\end{theorem}
\begin{proof}
	For any $ \alpha\in \mathbb{R}$  $ $ and $ \kappa_{1},\kappa_{2}\in K_{\alpha} $, then $ \kappa_{1},\kappa_{2}\in B $ and $ f(\kappa_{1})\leq\alpha, f(\kappa_{2})\leq\alpha $. Since $ B $ is a GLEI set, then there is a maximal positive number $ u(\kappa_{1},\kappa_{2})\leq1 $ such that $$ \gamma_{E(\kappa_{1}),E(\kappa_{2})}(t)\in B, \ \ \forall t\in\left[0,u(\kappa_{1},\kappa_{2}) \right] .$$ 
	In addition, since $ f $  is  GSLEP, there is a positive number $ v(\kappa_{1},\kappa_{2})\leq u(y_{1},y_{2}) $ such that 
	\begin{eqnarray}
	f\left(\gamma_{E(\kappa_{1}),E(\kappa_{2})}(t) \right)&\leq& t f(\kappa_{1}) +(1-t)f(\kappa_{2})\nonumber\\&\leq& t\alpha+(1-t)\alpha\nonumber\\ &=& \alpha, \ \ \forall t\in\left[0,v(\kappa_{1},\kappa_{2}) \right].\nonumber\end{eqnarray}
	That is , $ \gamma_{E(\kappa_{1}),E(\kappa_{2})}(t)\in K_{\alpha}, \ \ \forall t\in\left[0,v(\kappa_{1},\kappa_{2}) \right] $.
	Therefore, $ K_{\alpha} $ is a GLEI set with respect to $ \eta $ for any $ \alpha \in \mathbb{R} $.
\end{proof}
\begin{theorem}
	Let $ f:B\subset \aleph\longrightarrow \mathbb{R} $ where $ B $ is a GLEI. Then $ f $ is a  GSLEP  function with respect to $ \eta $ iff for each pair of points $ \kappa_{1},\kappa_{2}\in B $, there is a positive number $ v(\kappa_{1},\kappa_{2})\leq u(\kappa_{1},\kappa_{2})\leq 1 $ such that
	$$f\left(\gamma_{E(\kappa_{1}),E(\kappa_{2})}(t) \right) \leq t \alpha +(1-t)\beta , \ \ \forall t\in\left[0,v(\kappa_{1},\kappa_{2}) \right].$$
	\begin{proof}
		Let $ \kappa_{1},\kappa_{2}\in B $ and $ \alpha,\beta\in \mathbb{R} $ such that $ f(\kappa_{1})<\alpha $ and $ f(\kappa_{2})<\beta $. Since $ B $ is GLEI, there is a maximal positive number $ u(\kappa_{1},\kappa_{2})\leq 1 $ such that
		$$\gamma_{E(\kappa_{1}),E(\kappa_{2})}(t) \in B , \ \ \forall t\in\left[0,u(\kappa_{1},\kappa_{2}) \right].$$ 
		In addition, there is a positive number $ v(\kappa_{1},\kappa_{2})\leq u(\kappa_{1},\kappa_{2}) $ where 
		$$f\left(\gamma_{E(\kappa_{1}),E(\kappa_{2})}(t) \right) \leq t \alpha +(1-t)\beta , \ \ \forall t\in\left[0,v(\kappa_{1},\kappa_{2}) \right].$$
		Conversely, let $ (\kappa_{1},\alpha) \in \omega_{f} $ and $ (\kappa_{2},\beta) \in \omega_{f} $, then $ \kappa_{1},\kappa_{2}\in B $, $ f(\kappa_{1})<\alpha $ and $ f(\kappa_{2})<\beta $. Hence, $ f(\kappa_{1})<\alpha+\varepsilon  $ and $ f(\kappa_{2})<\beta+\varepsilon $ hold for any $ \varepsilon>0 $. According to the hypothesis for $ \kappa_{1},\kappa_{2}\in B $, there is a positive number $ v(\kappa_{1},\kappa_{2})\leq u(\kappa_{1},\kappa_{2})\leq 1 $ such that
		$$f\left(\gamma_{E(\kappa_{1}),E(\kappa_{2})}(t) \right) \leq t \alpha +(1-t)\beta+\varepsilon , \ \ \forall t\in\left[0,v(\kappa_{1},\kappa_{2}) \right].$$
		Let $ \varepsilon\longrightarrow0^{+} $, then 
		$$f\left(\gamma_{E(\kappa_{1}),E(\kappa_{2})}(t)\right) \leq t \alpha +(1-t)\beta , \ \ \forall t\in\left[0,v(\kappa_{1},\kappa_{2}) \right].$$
		That is 
		$\left(\gamma_{E(\kappa_{1}),E(\kappa_{2})}(t) , t \alpha +(1-t)\beta\right) \in \omega_{f} , \ \ \forall t\in\left[0,v(\kappa_{1},\kappa_{2}) \right].$\\
		Therefore, $ \omega_{f} $ is a GLEI set corresponding to $ \aleph $.
		From Theorem\ref{th1}, it follows that $ f $ is a GSLEP on $ B $ with respect to $ \eta $.
	\end{proof}	
\end{theorem}
\section{Optimality Criteria}
\hskip0.6cm
  
   \vspace{1.0mm} 
In this section, let us consider the nonlinear fractional multiobjective programming problem such as :\\
\begin{eqnarray*}
(VFP)  \begin{cases}
	minimize \frac{f(\kappa)}{g(\kappa)}=\left(\frac{f_{1}(\kappa)}{g_{1}(\kappa)},\cdots,\frac{f_{p}(\kappa)}{g_{p}(\kappa)} \right),\\  subject \ \ to\ \ h_{j}(\kappa)\leq0, j\in Q={1,2,\cdots q}  \\ \kappa\in K_{0}
	\end{cases}
\end{eqnarray*}
where $ K_{0}\subset \aleph $ is a  GLEI  set and $ g_{i}(\kappa)>0, \forall  \kappa\in K_{0} , i\in P={1,2,\cdots, p} $.

  \vspace{1.0mm}
Let $ f=(f_{1},f_{2},\cdots, f_{p}), g=(g_{1},g_{2},\cdots,g_{p}) $ and $ h=(h_{1},h_{2},\cdots,h_{q}) $\\ and denote $ K=\left\lbrace \kappa:h_{j}(\kappa)\leq 0, j\in Q, \kappa\in K_{0}\right\rbrace  $, the feasible set of problem ($ VFP $).\\ For $ \kappa^{*}\in K $, we put $ Q(\kappa^{*})=\left\lbrace j:h_{j}(\kappa^{*})= 0, j\in Q \right\rbrace $, $ L(\kappa^{*})=\frac{Q}{Q(\kappa^{*})} $.

  \vspace{1.0mm}
We also formulate the nonlinear multiobjective programming problem as follows:
\begin{eqnarray*}
	(VFP_{\lambda})  \begin{cases}
		minimize \left( f_{1}(\kappa)-\lambda_{1}g_{1}(\kappa),\cdots f_{p}(\kappa)-\lambda_{p}g_{p}(\kappa) \right),\\  subject \ \ to\ \ h_{j}(\kappa)\leq0, j\in Q={1,2,\cdots q}  \\ \kappa\in K_{0}
	\end{cases}
\end{eqnarray*}
where $  \lambda=(\lambda_{1},\lambda_{2},\cdots ,\lambda_{p})\in \mathbb{R}^{p} $.
  \vspace{1.0mm}
The followinng lemma connects the weak efficient solutions for ($ VFP $) and ($ VFP_{\lambda} $).
\begin{lemma}\label{Lemma1}
	A point $ \kappa^{*} $ is a weak efficient solution for ($ VFP_{\lambda} $) iff $ \kappa^{*} $ is a weak efficient solution for ($ VFP^{*}_{\lambda} $), where $ \lambda^{*}=(\lambda^{*}_{1},\cdots,\lambda^{*}_{p} )=\left(\frac{f_{1}(\kappa^{*})}{g_{1}(\kappa^{*})},\cdots,\frac{f_{p}(\kappa^{*})}{g_{p}(\kappa^{*})} \right) $.
\end{lemma}
\begin{proof}
	Assume that there is a feasible point $ \kappa\in K $, where 
	$$f_{i}(\kappa)-\lambda^{*}_{i}g_{i}(\kappa)<f_{i}(\kappa^{*})-\lambda^{*}_{i}g_{i}(\kappa^{*}),\forall i\in Q $$
	$ \Longrightarrow  $$$f_{i}(\kappa)<\frac{f_{i}(\kappa^{*})}{g_{i}(\kappa^{*})g_{i}(\kappa)}$$
	$ \Longrightarrow $ $$\frac{f_{i}(\kappa)}{g_{i}(\kappa)}<\frac{f_{i}(\kappa^{*})}{g_{i}(\kappa^{*})},$$
	which is a contradiction the weak efficiency of $ \kappa^{*} $ for ($ VFP $). 
	
	  \vspace{1.0mm}
	Now let us take $ \kappa\in K $ as a feasible point such that
	$$\frac{f_{i}(\kappa)}{g_{i}(\kappa)}<\frac{f_{i}(\kappa^{*})}{g_{i}(\kappa^{*})}= \lambda^{*}_{i},$$ then $ f_{i}(\kappa)-\lambda^{*}_{i}g_{i}(\kappa)<0=f_{i}(\kappa^{*})-\lambda^{*}_{i}g_{i}(\kappa^{*}), \forall i\in Q $,  which is agian contradiction to the weak efficiency of $ \kappa^{*} $ for ($ VFP^{*}_{\lambda} $). 
\end{proof} 
Next, some sufficient optimality conditions for the problem ($ VFP $) are established.
\begin{theorem}\label{th2}
Let $ \bar{\kappa}\in K, E(\bar{\kappa})=\bar{\kappa} $ and $ f,h $ be  GSLEP  and $ g $ be a  geodesic semilocal E-preincave, and they are all geodesic E-$ \eta $- semidifferentiable at $ \bar{\kappa} $. Further, assume that there are $ \zeta^{o}=\left(\zeta^{o}_{i}, i=1,\cdots,p \right)\in\mathbb{R}^{p}  $ and $ \xi^{o}=\left(\xi^{o}_{j}, j=1,\cdots,m \right)\in\mathbb{R}^{m}  $ such that 
\begin{equation}\label{eq4}
\zeta^{o}_{i}f'_{i+}\left(\gamma_{\bar{\kappa},E(\widehat{\kappa})}(t) \right)+\xi^{o}_{j} h'_{j+}\left(\gamma_{\bar{\kappa},E(\widehat{\kappa})}(t) \right)\geqslant 0\forall \kappa\in K, t\in[0,1], 
\end{equation}
\begin{equation}\label{eq5}
g'_{i+}\left(\gamma_{\bar{\kappa},E(\kappa)}(t) \right)\leq 0, \forall \kappa\in K, i\in P,
\end{equation}
\begin{equation}\label{eq6}
\xi^{o}h(\bar{\kappa})=0
\end{equation}
\begin{equation}\label{eq7}
\zeta^{o}\geqslant 0 , \xi^{o}\geqslant 0.
\end{equation}
Then $ \bar{\kappa} $ is a weak efficient solution for ($ VFP $).
\end{theorem}
\begin{proof}
	By contradiction, let $ \bar{\kappa} $ be not a weak efficient solution for ($ VFP $), then there exist a point $ \widehat{\kappa}\in K $ such that
	 \begin{equation}\label{eq8}
	\frac{f_{i}(\widehat{\kappa})}{g_{i}(\widehat{\kappa})}<\frac{f_{i}(\bar{\kappa})}{g_{i}(\bar{\kappa})}, i\in P.
	\end{equation}
	By the above hypotheses and Lemma \ref{Lemma1}, we have 
	\begin{equation}\label{eq9}
	f_{i}(\widehat{\kappa})-f_{i}(\bar{\kappa})\geqslant f'_{i+}\left(\gamma_{\bar{\kappa},E(\widehat{\kappa})}(t)\right) , i\in P
	\end{equation}
	\begin{equation}\label{eq10}
	g_{i}(\widehat{\kappa})-g_{i}(\bar{\kappa})\leq g'_{i+}\left(\gamma_{\bar{\kappa},E(\widehat{\kappa})}(t)\right) , i\in P
	\end{equation}
	\begin{equation}\label{eq11}
	h_{i}(\widehat{\kappa})-h_{i}(\bar{\kappa})\geqslant h'_{j+}\left(\gamma_{\bar{\kappa},E(\widehat{\kappa})}(t)\right) , j\in Q.
	\end{equation}
	Multiplying (\ref{eq9}) by $ \zeta^{o}_{i} $ and (\ref{eq11}) by $ \xi^{o}_{j} $, then we get
	\begin{eqnarray}\label{eq12}&&
	\sum_{i=1}^{p} \zeta^{o}_{i} \left(f_{i}(\widehat{\kappa})-f_{i}(\bar{\kappa}) \right) + \sum_{j=1}^{m} \xi^{o}_{j} \left(h_{j}(\widehat{\kappa})-h_{j}(\bar{\kappa}) \right)\nonumber\\&&\hspace{0.5in} \geqslant  \zeta^{o}_{i} f'_{i+}\left(\gamma_{\bar{\kappa},E(\widehat{\kappa})}(t)\right) +\xi^{o}_{j} h'_{j+}\left(\gamma_{\bar{\kappa},E(\widehat{\kappa})}(t)\right) \geqslant 0.
	\end{eqnarray}
Since $ \widehat{\kappa}\in K, \xi^{o}\geqslant 0 $ by (\ref{eq6}) and (\ref{eq12}), we have 
\begin{equation}\label{eq13}
\sum_{i=1}^{p} \zeta^{o}_{i} \left(f_{i}(\widehat{\kappa})-f_{i}(\bar{\kappa}) \right)\geqslant 0.
\end{equation}
Utilizing (\ref{eq7}) and (\ref{eq13}),then there is at least an $ i_{0} $ ($ 1\leq i_{0}\leq p $) such that 
\begin{equation}\label{eq14}
f_{i_{0}}(\widehat{\kappa})\geqslant f_{i_{0}}(\bar{\kappa}).
\end{equation}
On the other hand, (\ref{eq5}) and (\ref{eq10}) imply 
\begin{equation}\label{eq15}
g_{i}(\widehat{\kappa})\leq g_{i}(\bar{\kappa}), i\in P.
\end{equation}
By using (\ref{eq14}), (\ref{eq15}) and $ g>0 $, we have
\begin{equation}\label{eq16}
\frac{f_{i_{0}}(\widehat{\kappa})}{g_{i_{0}}(\widehat{\kappa})}\geqslant\frac{f_{i_{0}}(\bar{\kappa})}{g_{i_{0}}(\bar{\kappa})},
\end{equation}
which is a contradition with \ref{eq8}, then the proof of throrem is completed.
\end{proof}
Similarly we can prove the next theorem:
\begin{theorem}
	Consider that  $ \bar{\kappa}\in B, E(\bar{\kappa})=\bar{\kappa} $ and $ f,h $ are geodesic E-$ \eta $- semidifferentiable at $ \bar{\kappa} $. If there exist $ \zeta^{o}\in \mathbb{R}^{n} $ and $ \xi^{o}\in \mathbb{R}^{m} $ such that condition (\ref{eq4})-(\ref{eq7}) hold and $ \zeta^{o}f(x)+\xi^{o}h(x) $ is a GSLEP function, then $ \bar{\kappa} $ is a weak efficient solution for ($ VFP $).
\end{theorem}
\begin{theorem}
	Consider $ \bar{\kappa}\in B, E(\bar{\kappa})=\bar{\kappa} $ and $ \lambda_{i}^{o}=\frac{f_{i}(\bar{\kappa})}{g_{i}(\bar{\kappa})}(i\in P) $ are all pSLGEP functions and $ h_{j}(\kappa)(j\in \aleph(\bar{\kappa})) $ are all  GqSLEP functions and $ f,g,h $ are all geodesic E-$ \eta $-semidifferentiable at $ \bar{\kappa} $. If there is $ \zeta^{o}\in \mathbb{R}^{p} $ and $ \xi^{o}\in \mathbb{R}^{m} $ such that 
	\begin{eqnarray}\label{eq17}
	\sum_{i=1}^{p}\zeta_{i}^{o}\left(f'_{i+}\left( \gamma_{\bar{\kappa},E(\kappa)}(t)\right) -\lambda_{i}^{o}g'_{i+}\left( \gamma_{\bar{\kappa},E(\kappa)}(t)\right)  \right) +\xi^{o}h'_{i+}\left( \gamma_{\bar{\kappa},E(\kappa)}(t)\right)  \geqslant 0
	\end{eqnarray}
	\begin{eqnarray}\label{eq18}
	\xi^{o}h(\bar{\kappa})=0,
	\end{eqnarray}
	\begin{equation}\label{eq19}
	\zeta^{o}\geqslant 0, \xi^{o}\geqslant 0,
	\end{equation}
	then $ \bar{\kappa} $ is a weak efficient solution for ($ VFP $).
\end{theorem}
\begin{proof}
	 Assume that $ \bar{\kappa} $ is not a weak efficient solution for ($ VFP $). Therefore, there exists $ \kappa^{*}\in B $, yields 
	$$\frac{f_{i}(\kappa^{*})}{g_{i}(\kappa^{*})}<\frac{f_{i}(\bar{\kappa})}{g_{i}(\bar{\kappa})}.$$
Then
$$f_{i}(\kappa^{*})-\lambda_{i}^{o}g_{i}(\kappa^{*})<0,\ \ \ i\in P,$$ which means that 
$$f_{i}(\kappa^{*})-\lambda_{i}^{o}g_{i}(\kappa^{*})< f_{i}(\bar{\kappa})-\lambda_{i}^{o}g_{i}(\bar{\kappa})<0,\ \ \ i\in P. $$

By the pSLGEP  of $ \left( f_{i}(\kappa)-\lambda_{i}^{o}g_{i}(\kappa)\right) (i\in P) $ and Lemma \ref{lemma2}, we have 
 $$ f'_{i+}\left( \gamma_{\bar{\kappa},E(\kappa)}(t)\right) -\lambda_{i}^{o}g'_{i+}\left( \gamma_{\bar{\kappa},E(\kappa)}(t)\right) ,\ \ \ i\in P.  $$
 Utilizing $\zeta^{o}\geqslant 0  $, then
 \begin{equation}\label{eq20}
 \sum_{i=1}^{p}\zeta_{i}^{o}\left(f'_{i+}\left( \gamma_{\bar{\kappa},E(\kappa)}(t)\right) -\lambda_{i}^{o}g'_{i+}\left( \gamma_{\bar{\kappa},E(\kappa)}(t)\right)  \right)< 0.
 \end{equation}
 For $ h(\kappa^{*})\leq 0 $ and $ h_{j}(\bar{\kappa})= 0,\ \ \ j\in \aleph(\bar{\kappa}) $ , we have $ h_{j}(\kappa^{*})\leq h_{j}(\bar{\kappa}),\ \ \ \forall j\in \aleph(\bar{\kappa}). $
 
 By the GqSLEP of $ h_{j} $ and Lemma \ref{lemma2}, we have
 $$h_{j+}\left( \gamma_{\bar{\kappa},E(\kappa)}(t)\right) \leq 0,\ \ \ \forall j\in \aleph(\bar{\kappa}). $$
 Considering $ \xi^{o}\geqslant 0 $ and $ \xi_{j}^{o}= 0,\ \ \  j\in \aleph(\bar{\kappa}),  $ then
 \begin{equation}\label{eq21}
 \sum_{j=1}^{m}\xi_{j}^{o}h'_{j+}\left( \gamma_{\bar{\kappa},E(\kappa^{*})}(t)\right) \leq 0.
 \end{equation}
 Hence, by (\ref{eq20}) and (\ref{eq21}), we have 
 \begin{eqnarray}
 \sum_{i=1}^{p}\zeta_{i}^{o}\left(f'_{i+}\left( \gamma_{\bar{\kappa},E(\kappa^{*})}(t)\right) -\lambda_{i}^{o}g'_{i+}\left( \gamma_{\bar{\kappa},E(\kappa^{*})}(t)\right)  \right) +\xi^{o}h'_{i+}\left( \gamma_{\bar{\kappa},E(\kappa^{*})}(t)\right)  < 0,\nonumber\\
 \end{eqnarray}
 which is contradiction with relation (\ref{eq17}) at $ \kappa^{*}\in B $. 
 Therefore, $ \bar{\kappa} $ is a weak efficient solution for ($ VFP $).
\end{proof}
\begin{theorem}
Consider $ \bar{\kappa}\in B, E(\bar{\kappa})=\bar{\kappa} $ and $ \lambda_{i}^{o}=\frac{f_{i}(\bar{\kappa})}{g_{i}(\bar{\kappa})}(i\in P) $. Also, assume that  $ f,g,h $ are geodesic E-$ \eta $-semidifferentiable at $ \bar{\kappa} $. If there is $ \zeta^{o}\in \mathbb{R}^{p} $ and $ \xi^{o}\in \mathbb{R}^{m} $ such that the conditions (\ref{eq17})-(\ref{eq19}) hold and $ \sum_{i=1}^{p}\zeta^{o}_{i}\left(f_{i}(\kappa)-\lambda^{o}_{i}g_{i}(\kappa) \right)+\xi^{o}_{\aleph(\bar{\kappa})}h_{\aleph(\bar{\kappa})}(\kappa) $ is a  GpSLEP function, then $ \bar{\kappa} $ is a weak efficient soluion for ($ VFP $).
\end{theorem}
\begin{corollary}
Let $ \bar{\kappa}\in B, E(\bar{\kappa})=\bar{\kappa} $ and $ \lambda_{i}^{o}=\frac{f_{i}(\bar{\kappa})}{g_{i}(\bar{\kappa})}(i\in P) $. Futher let $ f, h_{\aleph(\bar{\kappa})} $ be all GSLEP function, $ g $ be a geodesic semilocal E-preincave function and $ f,g,h $ be all geodesic E-$ \eta $- semidifferentiable at $ \bar{\kappa} $. If there exist $ \zeta^{o}\in \mathbb{R}^{p} $ and $ \xi^{o}\in \mathbb{R}^{m} $ such that the conditions (\ref{eq17})-(\ref{eq19}) hold, then $ \bar{\kappa} $ is a weak efficient soluion for ($ VFP $). 
\end{corollary}
	  \vspace{1.0mm}
  The dual problem for  ($ VFP $) is formulated as follows
\begin{eqnarray*}
	(VFD)  \begin{cases}
		minimize \left(\zeta_{i}, i=1,2,\cdots, p \right) ,\\  subject \ \ to\ \ \sum_{i=1}^{p}\alpha_{i}\left(f'_{i+}\left( \gamma_{\lambda,E(\kappa)}(t)\right) -\zeta_{i}g'_{i+}\left( \gamma_{\lambda,E(\kappa)}(t)\right)  \right)\\\hspace{1.5in}  +\sum_{j=1}^{m}\beta_{j}h'_{j+}\left( \gamma_{\lambda,E(\kappa)}(t)\right)  \geqslant 0 \\ \kappa\in K_{0}, t\in[0,1],\\
		f_{i}(\lambda)-\zeta_{i}g_{i}(\lambda)\geqslant0,\ \ \ i\in P,
		\beta_{j}h_{j}(\lambda)\geqslant 0,\ \ \ j\in \aleph,\\
	\end{cases}
\end{eqnarray*}
where
$\zeta =(\zeta_{i}, i=1,2,\cdots, p)\geqslant 0$, $\alpha =(\alpha_{i}, i=1,2,\cdots, p)> 0$,\\ $\beta =(\beta_{i}, i=1,2,\cdots, m)\geqslant 0$,   $\lambda\in K_{0}. $

\vspace{1.0mm}
Denote the feasible set problem ($ VFD $) by $ K^{,} $.
\begin{theorem}[General Weak Duality]
Let $ \kappa\in K $, $ (\alpha,\beta,\lambda,\zeta)\in K^{,} $ and $ E(\lambda)=\lambda $. If $ \sum_{i=1}^{p}\alpha_{i}(f_{i}-\zeta_{i}g_{i}) $ is a  GpSLEP  function and $ \sum_{j=1}^{m}\beta_{j}h_{j} $ is a GqSLEP  function and they are all geodesic E-$ \eta $-semidifferentiable at $ \lambda $, then $ \frac{f(\kappa)}{g(\kappa)}\nleq \zeta $.
\end{theorem}
\begin{proof}
From $ \alpha>0 $ and $ (\alpha, \beta,\lambda,\zeta)\in K^{,} $, we have 
$$\sum_{i=1}^{p}\alpha_{i}(f_{i}(\kappa)-\zeta_{i}g_{i}(\kappa))<0\leq\sum_{i=1}^{p}\alpha_{i}(f_{i}(\lambda)-\zeta_{i}g_{i}(\lambda)). $$
By the GpSLEP of $ \sum_{i=1}^{p}\alpha_{i}(f_{i}-\zeta_{i}g_{i}) $ and Lemma \ref{lemma2}, we obtain $$\left( \sum_{i=1}^{p}\alpha_{i}(f_{i}-\zeta_{i}g_{i})  \right)'_{+}\left(\gamma_{\lambda,E(\kappa)}(t) \right) <0, $$ that is, 
$$\sum_{i=1}^{p}\alpha_{i}\left(f'_{i+}\left( \gamma_{\lambda,E(\kappa)}(t)\right) -\zeta_{i}g'_{i+}\left( \gamma_{\lambda,E(\kappa)}(t)\right]  \right)<0. $$
Also, from $ \beta\geqslant  0$ and $ \kappa\in K $, then
$$\sum_{j=1}^{m}\beta_{j}h_{j}(\kappa)\leq 0 \leq\sum_{j=1}^{m}\beta_{j}h_{j}(\lambda). $$ 
Using the GqSLEP of $ \sum_{j=1}^{m}\beta_{j}h_{j} $ and Lemma \ref{lemma2}, one has
$$\left( \sum_{j=1}^{m}\beta_{j}h_{j} \right)'_{+}\left(\gamma_{\lambda,E(\kappa)}(t) \right) \leq 0. $$
Then 
$$ \sum_{j=1}^{m}\beta_{j}h'_{j+} \left(\gamma_{\lambda,E(\kappa)}(t) \right) \leq 0. $$
Therefore 
$$\sum_{i=1}^{p}\alpha_{i}\left(f'_{i+}\left( \gamma_{\lambda,E(\kappa)}(t)\right) -\zeta_{i}g'_{i+}\left( \gamma_{\lambda,E(\kappa)}(t)\right)  \right) +\sum_{j=1}^{m}\beta_{j}h'_{j+}\left( \gamma_{\lambda,E(\kappa)}(t)\right)  < 0, $$
This is a contradiction with $ (\alpha,\beta,\lambda,\zeta)\in K^{,} $. 
\end{proof}
\begin{theorem}
	Consider that $ \kappa\in K $, $ (\alpha, \beta,\lambda,\zeta)\in K^{,} $ and $ E(\lambda)=\lambda $. If $ \sum_{i=1}^{p}\alpha_{i}(f_{i}-\zeta_{i}g_{i})+\sum_{j=1}^{m}\beta_{j}h_{j}  $ is a  GpSLEP function  and geodesic E-$ \eta $-semidifferentiable at $ \lambda $, then $ \frac{f(\kappa)}{g(\kappa)}\nleq \zeta $.
\end{theorem}
\begin{theorem}[General Converse Duality]
Let $ \bar{\kappa}\in K $ and $ (\kappa^{*},\alpha^{*}, \beta^{*},\zeta^{*})\in K^{,} $,$ E(\kappa^{*})=\kappa^{*} $, where $\zeta^{*}= \frac{f(\kappa^{*})}{g(\kappa^{*})}=\frac{f(\bar{\kappa})}{g(\bar{\kappa})}=(\zeta^{*}_{i}, \ \ \ i=1,2,\cdots, p) $. If $ f_{i}-\zeta_{i}^{*}g_{i} (i\in P), h_{j}(j\in \aleph)$  are all GSLEP  functions and all geodesic E-$ \eta $-semidifferentiable at $ \kappa^{*} $, then $ \bar{\kappa} $ is a weak efficient solution for ($ VFP $).
\end{theorem}
\begin{proof}
	By using the hypotheses and Lemma \ref{lemma2}, for any $ \kappa\in K $, we obtain  $$\left( f_{i}(\kappa)-\zeta_{i}^{*}g_{i}(\kappa)\right) -\left(f_{i}(\kappa^{*})-\zeta_{i}^{*}g_{i}(\kappa^{*}) \right)\geqslant f'_{i+}\left( \gamma_{\kappa^{*},E(\kappa)}(t)\right) -\zeta_{i}g'_{i+}\left( \gamma_{\kappa^{*},E(\kappa)}(t)\right)  $$
	$$h_{j}(y)-h_{j}(\kappa^{*})\geqslant h'_{j+}\left( \gamma_{\kappa^{*},E(\kappa)}(t)\right). $$
	
	\vspace{1.0mm}
	Utilizing the fiest constraint condition for ($ VFD $), $ \alpha^{*}>0,\beta^{*}\geqslant 0, \zeta^{*}\geqslant 0 $ and the two inequalilities above, hence
	\begin{eqnarray}&&
	\sum_{i=1}^{p}\alpha^{*}_{i}\left(\left( f_{i}(\kappa)-\zeta_{i}^{*}g_{i}(\kappa)\right) -\left(f_{i}(\kappa^{*})-\zeta_{i}^{*}g_{i}(\kappa^{*}) \right) \right) + \sum_{j=1}^{m}\beta^{*}_{j}\left(h_{j}(\kappa)-h_{j}(\kappa^{*}) \right) \nonumber\\\hspace{0.5in} &\geqslant& \sum_{i=1}^{p}\left(f'_{i+}\left( \gamma_{\kappa^{*},E(\kappa)}(t)\right) -\zeta_{i}g'_{i+}\left( \gamma_{\kappa^{*},E(\kappa)}(t)\right) \right)\nonumber\\\hspace{0.5in}&+&\sum_{j=1}^{m}\beta^{*}_{j} h'_{j+}\left( \gamma_{\kappa^{*},E(\kappa)}(t)\right)\nonumber\\\hspace{0.5in} &\geqslant& 0.
	\end{eqnarray}
	In view of $ h_{j}(\kappa)\leq 0, \beta^{*}_{j}\geqslant 0, \beta^{*}_{j}h_{j}(\kappa^{*})\geqslant (j\in \aleph)  $ and $ \zeta^{*}_{i}= \frac{f_{i}(\kappa^{*})}{g_{i}(\kappa^{*})}\ \ \ (i\in P) $, then
	\begin{equation}\label{eq22}
	\sum_{i=1}^{p}\alpha^{*}_{i}\left( f_{i}(\kappa)-\zeta_{i}^{*}g_{i}(\kappa)\right)\geqslant 0 \ \ \ \forall y\in Y .
	\end{equation} 
	
	Consider that $ \bar{\kappa} $ is not a weak efficient solution for ($ VFP $). From $ \zeta^{*}_{i}= \frac{f_{i}(\bar{\kappa})}{g_{i}(\bar{\kappa})}\ \ \ (i\in P) $ and Lemma \ref{Lemma1}, it follows that $ \bar{\kappa} $ is not a weak efficient solution for ($ VFP_{\zeta^{*}} $). Hence, $ \tilde{\kappa}\in K $ such that 
	$$ f_{i}(\tilde{\kappa})-\zeta_{i}^{*}g_{i}(\tilde{\kappa}) <f_{i}(\bar{\kappa})-\zeta_{i}^{*}g_{i}(\bar{\kappa}) = 0,\ \ \ i\in P,  $$
	hence 
	$ \sum_{i=1}^{p}\alpha^{*}_{i}\left(f_{i}(\tilde{\kappa})-\zeta_{i}^{*}g_{i}(\tilde{\kappa}) \right)<0  $.
	This is a contradiction to the inequality (\ref{eq22}). The proof of theorem is completed. 
\end{proof}


\begin{thebibliography}{15}
	
	\bibitem{kS}  Adem K{\i}l{\i}\c{c}man and Wedad Saleh ,\textit{ Some Inequalities For Generalized s-Convex Functions}, JP Journal of Geometry and Topology, \textbf(17)(2015), 63-82.
	
	\bibitem{AW}  Adem K{\i}l{\i}\c{c}man and Wedad Saleh, \textit{On Geodesic Strongly E-convex Sets and Geodesic Strongly E-convex Functions}.Journal of Inequalities and Applications, \textbf(2015)1 (2015) 1-10.
\bibitem{Agarwal}   R. P. Agarwal, I. Ahmad, A. Iqbal, and S. Ali, \textit{Generalized invex sets and preinvex functions on Riemannian manifolds}.Taiwanese Journal of Mathematics, \textbf(16)5(2012) 1719–1732.

	
\bibitem{BG} R. Bartolo, A. Germinario, and M. S\'{a}nchez,\textit {Convexity of domains of Riemannian manifolds}. Annals of Global Analysis
and Geometry, \textbf(21)1(2002) 63–83.	
	

\bibitem{Chavel1993} I. Chavel,\textit{Riemannian Geometry—A Modern Introduction}, Cambridge University Press, Cambridge, UK, 1993.
	
	
\bibitem{Ewing}G.M. Ewing,\textit{Sufficient conditions for global minima of suitably convex functions from variational and control theory}.
SIAM Rev. \textbf(19) (1977) 202–220. doi:10.1137/1019037	
	
	
	
\bibitem{Ferrara}  M. Ferrara and S. Mititelu, \textit{Mond-Weir duality in vector programming with generalized invex functions on differentiable
manifolds}.BalkanJournal of Geometry and its Applications, \textbf(11) 2(2006) 80–87.
	
	\bibitem{FulgaPreda} C. Fulga and V. Preda,\textit{ Nonlinear programming with E-preinvex and local E-preinvex functions}. European Journal of Operational Research, \textbf(192)3 (2009)737-743 .
	
	\bibitem{GrinalattLinnainmaa2011} M.Grinalatt and J.T.Linnainmaa, \textit{ Jensen's inquality, parameter uncertainty, and multiperiod investment}.Review of Asset Pricing Studies,\textbf(1),1 (2011) ,1-34.
	
 
 
 
\bibitem{Hu2007}Q.J. Hu, J.B. Jian, H.Y.  Zheng and C.M. Tang \textit{ Semilocal E-convexity and semilocal E-convex programming}. Bull Aust Math Soc. \textbf(75)(2007)59–74. doi:10.1017/S0004972700038983.


	
	\bibitem{Iqbal}A.Iqbal, I.Ahmad and S.Ali,\textit{ Some properties of geodesic semi-E-convex functions}. Nonlinear Analysis: Theory, Method and Application, \textbf(74)17 (2011)6805-6813.
	
	\bibitem{IAA2012} A.Iqbal, S.Ali and I.Ahmad,\textit{ On geodesic E-convex sets, geodesic E-convex functions and E-epigraphs}. J.Optim Theory Appl.\textbf(55)1(2012)239-251.
	
	\bibitem{Jiao2011} H. Jiao  and S. Liu, \textit{Semilocal E-preinvexity and its applications in nonlinear multiple objective fractional programming}. Journal of Inequalities and Applications \textbf(2011)1 (2011): 116.	

\bibitem{Jiao2012}   H. Jiao, S. Liu, and X. Pai, \textit{A Class of Semilocal E-Preinvex Functions and Its Applications in Nonlinear Programming}. Journal of Applied Mathematics, \textbf( 2012) 2012.
	
\bibitem{Jiao2013}  H. Jiao, \textit{ A class of semilocal E-preinvex maps in Banach spaces with applications to nondifferentiable vector optimization}. An International Journal of Optimization and Control: Theories \& Applications (IJOCTA), \textbf(4)1 (2013) 1-10.	
	
\bibitem{Klingenberg1982}  W. Klingenberg,\textit{ Riemannian Geometry}, vol. 1 of Walter de Gruyter Studies in Mathematics, Walter de Gruyter, Berlin,Germany, 1982.	
	
\bibitem{Mordukhovch2011}C. Li, B. S. Mordukhovich, J.Wang, and J.-C. Yao,\textit {Weak sharp minima on Riemannian manifolds.} SIAM Journal onOptimization, \textbf(21) 4(2011) 1523–1560.	
	

\bibitem{Niculescu2007Optimality}  C. Niculescu,\textit{ Optimality and duality in multiobjective fractional programming involving ρ-semilocally type I-preinvex and related functions}. Journal of mathematical analysis and applications, \textbf{335}1 (2007) 7-19.

	

\bibitem{Preda2003}V. Preda \textit{Optimality and duality in fractional multiple objective programming involving semilocally preinvex and related
	functions}. J Math Anal Appl. \textbf(288)(2003) 365–382. doi:10.1016/S0022-247X(02)00460-2


\bibitem{Preda1997} V.Preda and I.M. Stancu-Minasian, \textit{Duality in multiple objective programming involving semilocally preinvex and related
	functions}. Glas Mat Ser Ill. \textbf(32)(1997)153–165.

\bibitem{Rapcsak1997}  T. Rapcs\'{a}k,\textit{ Smooth Nonlinear Optimization in $ \mathbb{R}^{n} $ }, Kluwer Academic, Dordrecht,The Netherlands, 1997.

	\bibitem{RuelAyres1999} J.J.Ruel and M.P.Ayres, \textit{Jensen's inequality predicts effects of environmental variation}. Trends in Ecology and Evolution,\textbf(14)9 (1999),361--366.
	

	\bibitem{Udriste1994}  C. Udriste,\textit{ Convex functions and optimization methods on Riemannian manifolds}, Mathematicsand Its Applications, \textbf(297), Kluwer Academic, Boston, Mass, USA, 1994.
	



	

	
	
 \end{thebibliography}
\end{document}